\documentclass[sn-mathphys,Numbered]{sn-jnl}

\usepackage{graphicx}%
\usepackage{multirow}%
\usepackage{amsmath,amssymb,amsfonts}%
\usepackage{amsthm}%
\usepackage{mathrsfs}%
\usepackage[title]{appendix}%
\usepackage{xcolor}%
\usepackage{textcomp}%
\usepackage{manyfoot}%
\usepackage{booktabs}%
\usepackage{algorithm}%
\usepackage{algorithmicx}%
\usepackage{algpseudocode}%
\usepackage{listings}%

\theoremstyle{thmstyleone}%
\newtheorem{theorem}{Theorem}
\newtheorem{proposition}[theorem]{Proposition}%

\theoremstyle{thmstyletwo}%
\newtheorem{corollary}{Corollary}%
\newtheorem{lemma}{Lemma}%

\theoremstyle{thmstylethree}%

\begin{document}

\title[Article Title]{Tests for the mean of high-dimensional data}


\author{\fnm{Dietmar} \sur{Ferger}}\email{dietmar.ferger@tu-dresden.de}



\affil{\orgdiv{Fakult\"{a}t Mathematik}, \orgname{Technische Universit\"{a}t Dresden}, \orgaddress{\street{Zellescher Weg 12-14}, \city{Dresden}, \postcode{01069}, \country{Germany}}}




\abstract{We consider the problem of testing the mean of high-dimensional data when the dimension may grow without explicit rate restrictions relative to the sample size. The proposed procedure is based on the statistic $V_n=n||\overline{X}_n||^2$,
which avoids inversion of the covariance matrix and is therefore suitable for high-dimensional settings. We establish asymptotic distributional results for both fixed and increasing dimension by embedding the observations into the Hilbert space $l^2$.
Furthermore, we prove the asymptotic validity of a bootstrap approximation for the distribution of the test statistic. The resulting bootstrap test yields asymptotic level-a procedures without requiring sparsity assumptions or structural conditions on the covariance matrix.
In all this, a new Central Limit Theorem in $l^2$ is proving to be an extremely useful tool.}




\maketitle

\section{Introduction}
The problem of testing hypotheses on high-dimensional mean vectors has attracted considerable attention in recent years, motivated by applications in fields such as genomics, finance, and signal processing. In modern data-sets, the dimension of the observations is often comparable to or larger than the sample size, which renders classical multivariate procedures unreliable. In particular, Hotelling's $T^2$
-test is no longer applicable when the sample covariance matrix becomes singular, a phenomenon that occurs as soon as the dimension exceeds the sample size. To overcome this difficulty, a large body of literature has developed alternative procedures that avoid explicit inversion of the covariance matrix. A seminal contribution is due to Bai and Saranadasa \cite{Bai}, who proposed a test based on the squared Euclidean norm of the sample mean. This idea has been further developed by Chen and Qin \cite{ChenQin} and many subsequent authors (e.g., Srivastava and Du \cite{SrivastavaDu}, Zhang et al. \cite{Zhang}), leading to a class of $L_2$-type tests that are particularly powerful under dense alternatives.

Complementary approaches are based on maximum-type statistics and multiple testing ideas, which are well suited for sparse alternatives. Notable examples include the procedures of Cai, Liu and Xia \cite{Cai} and Xue and Yao \cite{XueYao}, as well as higher criticism methods introduced by Donoho and Jin \cite{DonohoJin}. Adaptive procedures that combine the advantages of different test statistics have also been proposed (e.g., Xu et al. \cite{Xu}, He et al. \cite{He}), reflecting the fact that no single method is uniformly optimal across different alternatives.
Further developments include projection-based tests that incorporate dependence structures (Lopes et al. \cite{Lopes} , Thulin \cite{Thulin}, Huang \cite{Huang}), as well as nonparametric and robust methods based on spatial signs, ranks, or U-statistics (Chakraborty and Chaudhuri \cite{Chakraborty}, He et al. \cite{He}). Despite this diversity, most existing approaches rely on specific structural assumptions, such as sparsity, dependence patterns, or explicit growth conditions relating the dimension d to the sample size n.\\

In contrast to these frameworks for two-sample problems, we consider for the one-sample problem an asymptotic regime in which the dimension $d_n$ is allowed to diverge without imposing explicit rate conditions relative to the sample size. Our approach is based on the test statistic
$$
 V_n := n ||\overline{X_n}||^2,
$$
which arises naturally from the Euclidean norm of the sample mean. In fixed dimension, the asymptotic distribution of $V_n$
follows from the multivariate Central Limit Theorem and can be expressed as a weighted sum of independent $\chi_1^2$-random variables. However, this representation depends on the eigenvalues of the covariance matrix, which are typically unknown and difficult to estimate in high-dimensional settings.\\

To overcome this difficulty, we employ bootstrap methods to approximate the distribution of the test statistic. Our analysis proceeds in two steps. First, in Lemma \ref{convVn} we consider the finite-dimensional case and establish a triangular array version of the Central Limit Theorem for the statistic $V_n$.
In particular, it forms the basis for subsequent results. Building on this, Theorem 1 establishes the consistency of the bootstrap approximation in fixed dimension.
In a second step, we extend the analysis to the high-dimensional setting where $d_n \rightarrow \infty.$
To this end, we embed the observations into the Hilbert space $l^2$
and interpret the finite-dimensional vectors as truncations of an underlying infinite-dimensional random element. Within this framework, Proposition \ref{highdimensions} establishes weak convergence of $V_n=V_n(d_n)$ under suitable second-moment and covariance conditions. The key technical result is Proposition \ref{CLTl2arrays}, which provides a new central limit theorem for triangular arrays of $l^2$-valued random variables. This result can be viewed as a concrete and verifiable version of abstract central limit theorems for Banach space-valued random variables, tailored to the structure of $l^2$.

Finally, we combine these results to establish the asymptotic validity of the bootstrap in the high-dimensional setting. This yields a practical procedure for constructing tests based on $V_n(d_n)$ without requiring explicit estimation of the eigenvalues of the covariance operator. Our results thus provide an alternative approach to high-dimensional mean testing that avoids sparsity assumptions and explicit rate conditions, and instead exploits the underlying Hilbert space structure of the data. As for $d$ finite (Proposition \ref{levelalphatestdfinite}) we obtain asymptotic level-$\alpha$ tests (Proposition \ref{levelalphatestdinfinite}).

\section{The case $d$ is finite}
For each $n \in \mathbb{N}$ let $X_1,\ldots,X_n$ be independent and identically distributed (i.i.d.) random vectors in the euclidian space $\mathbb{R}^d$
defined on a common probability space $(\Omega,\mathcal{A},\mathbb{P})$. If $\mu:=\mathbb{E}[X_1]$, then we want to test the hypothesis $H_0:\mu=\mu_0$ against the alternative $H_1:\mu \neq \mu_0$ for some given $\mu_0 \in \mathbb{R}^d$. Since $\mu_0$ is known, we can w.l.o.g. assume that $\mu_0=0$, because otherwise transform the data into the centered $X_i-\mu_0$.
Notice that $\mu=0$ or $\mu \neq 0$ is equivalent to $||\mu||^2 = 0$ or $||\mu||^2>0$, respectively.
Here, $||x||:=\sqrt{x^\prime x}$ denotes the euclidian norm of a column-vector $x \in \mathbb{R}^d$, where the superscript $\prime$ denotes matrix transposition. Since the arithmetic
mean $\overline{X}_n :=n^{-1}\sum_{i=1}^n X_i$ is an unbiased and consistent estimator for $\mu$, it is reasonable to reject the hypothesis $H_0$ for large values of $||\overline{X}_n||^2$. Assume that $\mathbb{E}[||X_1||^2] < \infty.$ Then the multivariate Central Limit Theorem (CLT) says that under $H_0$
\begin{equation} \label{CLT}
   \frac{1}{\sqrt{n}}\sum_{i=1}^n X_i \stackrel{\mathcal{D}}{\rightarrow} Z \sim N(0,\Gamma)\quad \text{in } \mathbb{R}^d,
\end{equation}
where $\Gamma = \text{Cov}(X_1)$ is the covariance matrix of $X_1$. Thus
\begin{equation} \label{VCLT}
 V_n := n ||\overline{X_n}||^2=||\frac{1}{\sqrt{n}} \sum_{i=1}^n X_i||^2 \stackrel{\mathcal{D}}{\rightarrow} ||Z||^2 =:V
\end{equation}
by continuity of the norm $||\cdot||$ and the Continuous Mapping Theorem (CMT).
As to the limit variable $V$ in (\ref{VCLT}) it is well known that
$$
 V \stackrel{\mathcal{D}}{=} \sum_{i=1}^d \lambda_i N_i^2,
$$
where $\lambda_1,\ldots,\lambda_d$ are the eigenvalues of $\Gamma$ and $N_1,\ldots,N_d$ are i.i.d. standard normal random variables.
In particular, the distribution function $F$ of $V$ is continuous. Therefore,
by P\'{o}lya's theorem the distributional convergence (\ref{VCLT}) is equivalent to
\begin{equation} \label{uniformVn}
 \sup_{x \in \mathbb{R}}|\mathbb{P}(V_n \le x)-F(x)| \rightarrow 0, \;  n \rightarrow \infty.
\end{equation}

Now, we reject $H_0$, if $V_n > c_\alpha$, where $c_\alpha$ is the $(1-\alpha)$-quantile of $F$ with $\alpha \in (0,1)$. It follows from (\ref{VCLT}) and continuity of $F$ that this is an asymptotic level-$\alpha$ test. Even though we have a nice representation of the limit variable $V$ it is unknown to the statistician. One obvious way would be to estimate the unknown covariance matrix $\Gamma$ by the sample covariance matrix, say $\Gamma_n$, and then approximate the $\lambda_i$ by the eigenvalues $\lambda_{n,i}$ of $\Gamma_n$. This leads to
$\sum_{i=}^d \lambda_{n,i} N_i^2$ and the $(1-\alpha)$-quantiles thereof as reasonable estimates for the critical values $c_\alpha$. However,
when $p$ is very large, the procedure becomes intractable (curse of high dimension). As a way out we will approximate the critical values by a
resampling method (Efron's bootstrap). A basic tool for showing that our \emph{bootstrap works} (see Corollary \ref{bootworks} below) is the following lemma. It follows, when in our above derivation the CLT (\ref{CLT}) for sequences is replaced by the CLT of Lindeberg-Feller for arrays of random vectors.\\

\begin{lemma} \label{convVn} For every $n \in \mathbb{N}$ let $X_{n1},\ldots,X_{nn}$ be i.i.d. random vectors in $\mathbb{R}^d$ with
$\mathbb{E}[X_{n1}]=0$ and $\mathbb{E}[||X_{n1}||^2]< \infty$. If
\begin{equation}
 \text{Cov}[X_{n1}] \rightarrow \Gamma
\end{equation}
and
\begin{equation} \label{Lindebergconditionwithrationals}
 L_n(\epsilon) := \mathbb{E}[||X_{n1}||^2 1_{\{||X_{n1}||> \sqrt{n} \epsilon\}}] \rightarrow 0 \quad \text{for all positive rational }  \epsilon,
\end{equation}
then $\overline{X}_n := n^{-1} \sum_{i=1}^n X_{ni}$ satisfies the following limit-law:
$$
 V_n:= n ||\overline{X}_n||^2 \stackrel{\mathcal{D}}{\rightarrow} V=\sum_{i=1}^p \lambda_i N_i^2.
$$
Here, $\lambda_1,\dots,\lambda_d$ are the eigenvalues of $\Gamma$ and $N_1,\ldots,N_d$ are i.i.d. standard normal. As we know
this is equivalent to
$$
  \sup_{x \in \mathbb{R}}|\mathbb{P}(V_n \le x)-F(x)| \rightarrow 0, \;  n \rightarrow \infty,
$$
where $F$ is the distribution function of $V$.
\end{lemma}

\vspace{0.5cm}
The proof of the Lindeberg-Feller CLT shows that the Lindeberg-condition (\ref{Lindebergconditionwithrationals}) with \emph{rational} $\epsilon >0$  is sufficient. In the next section we will use countability of the rational numbers $\mathbb{Q}.$

\section{The bootstrap when $d$ is finite}
For every $n \in \mathbb{N}$ let $i_{n1},\ldots,i_{nn}$ be random variables defined on some probability space $(\Omega^*,\mathcal{A}^*,\mathbb{P}^*)$,
which are $\mathbb{P}^*$-independent and each of these is uniformly distributed on $\{1,\ldots,n\}$, i.e. for every $1 \le k \le n$, the following holds:
$$
  \mathbb{P}^*(\{\omega^* \in \Omega^*: i_{nk}(\omega^*)=j\})=\frac{1}{n} \quad \forall \; j \in \{1,\ldots,n\}.
$$
For each sample point $\omega \in \Omega$ we define random variables $\xi_{nk}(\omega):\Omega^* \rightarrow \mathbb{R}^d$ by
$\xi_{nk}(\omega)(\omega^*):=X_{i_{nk}(\omega^*)}(\omega), \omega^* \in \Omega^*,$ for all $1 \le k \le n.$ A straightforward calculation shows that
$\xi_{n1}(\omega),\ldots,\xi_{nn}(\omega)$ are $\mathbb{P}^*$-independent with common distribution $Q_n(\omega)$, where
\begin{equation} \label{defQn}
 Q_n(\omega):= n^{-1} \sum_{i=1}^n \delta_{X_i(\omega)}
\end{equation}
is the empirical measure pertaining to the realizations $X_1(\omega),\ldots,X_n(\omega)$. Here, $\delta_x$ denotes the Dirac-measure at point $x \in \mathbb{R}^d$. In particular, the following applies:
\begin{equation} \label{distribbootstrap}
\mathbb{P}^*\circ (\xi_{n1}(\omega),\ldots,\xi_{nn}(\omega))^{-1}=Q_n(\omega)\otimes \cdots \otimes Q_n(\omega).
\end{equation}

If $\overline{X}_n(\omega):= n^{-1} \sum_{i=1}^n X_i(\omega)$, then
with $\overline{\xi}_n(\omega):= n^{-1} \sum_{i=1}^n (\xi_{ni}(\omega)- \overline{X}_n(\omega))$
the bootstrap version of $V_n$ is given by
$$
 V_n^*(\omega):= n ||\overline{\xi}_n(\omega)||^2.
$$
Recall that for each $\omega \in \Omega$ the test-statistic $V_n^*(\omega)$ is a random variable defined on the probability space $(\Omega^*,\mathcal{A}^*,\mathbb{P}^*)$ with possible values $V_n^*(\omega)(\omega^*), \omega^* \in \Omega^*$.\\

\begin{theorem} \label{bootstrapfinite} Suppose $X_i, i \in \mathbb{N},$ are i.i.d. random variables with values in $\mathbb{R}^d$ such that
$\mathbb{E}[X_1]=0$ and  $\mathbb{E}[||X_1||^2] < \infty$. Then
\begin{equation} \label{bootstrapconv}
 V_n^*(\omega)= n ||\overline{\xi}_n(\omega)||^2 \stackrel{\mathcal{D}}{\rightarrow} V=\sum_{i=1}^d \lambda_i N_i^2 \quad \text{for } \mathbb{P}\text{-almost all } \omega \in \Omega.
\end{equation}
Here, $\lambda_1,\dots,\lambda_d$ are the eigenvalues of $\Gamma = \text{Cov}[X_1]$ and $N_1,\ldots,N_d$ are i.i.d. standard normal. In fact, uniform convergence holds:
\begin{equation} \label{uniformVnstar}
  \sup_{x \in \mathbb{R}}|\mathbb{P}^*(V_n^*(\omega) \le x)-F(x)| \rightarrow 0, \;  n \rightarrow \infty,\quad \text{for } \mathbb{P}\text{-almost all } \omega \in \Omega.
\end{equation}
where $F$ is the distribution function of $V$.
\end{theorem}

\begin{proof} For each fixed $\omega \in \Omega$ the array $X_{ni}(\omega):=\xi_{ni}(\omega)- \overline{X}_n(\omega), 1 \le i \le n, n \in \mathbb{N},$ consists of random variables $X_{ni}(\omega)$, which are defined on $(\Omega^*,\mathcal{A}^*,\mathbb{P}^*)$ and which by (\ref{distribbootstrap}) are rowwise i.i.d. under $\mathbb{P}^*$. In the sequel we check that the array $(X_{ni}(\omega): 1 \le i \le n, n \in \mathbb{N})$ satisfies the remaining conditions of Lemma \ref{convVn} for $\mathbb{P}-$almost all $\omega \in \Omega.$ To start with
$$
   \mathbb{E}^*[\xi_{ni}(\omega)]=\int_{\Omega^*} \xi_{ni}(\omega)(\omega^*) \mathbb{P}^*(d\omega^*)= \int_{\mathbb{R}^p} x  Q_n(\omega)(dx)=n^{-1} \sum_{i=1}^n X_i(\omega) = \overline{X}_n(\omega),
$$
where the second equality holds by Change of Variable plus (\ref{distribbootstrap}) and the third equality follows from the definition (\ref{defQn}) of $Q_n(\omega)$. Since $\overline{X}_n(\omega)$ is a constant with respect to $\omega^* \in \Omega^*$, it follows that
$\mathbb{E}^*[X_{ni}(\omega)]=\overline{X}_n(\omega)-\overline{X}_n(\omega)=0.$ Similarly, $\mathbb{E}^*[||X_{n1}(\omega)||^2]=n^{-1} \sum_{i=1}^n ||X_i(\omega)-\overline{X}_n(\omega)||^2$, which as a finite sum of real numbers is finite. In the sequel we use the following notations:
If $x \in \mathbb{R}^d$ is a (column-) vector, then $x^{(k)}$ denotes its $k$-th component. If $M$ is a matrix, then $M_{kl}$ is its $(k,l)$-th entry. It follows that
\begin{eqnarray*}
& &(\text{Cov}^*[X_{n1}(\omega)])_{kl}= \text{Cov}^*[X_{n1}(\omega)^{(k)},X_{n1}(\omega)^{(l)}]=\text{Cov}^*[\xi_{n1}(\omega)^{(k)},\xi_{n1}(\omega)^{(l)}]\\
&=& E^*[\xi_{n1}(\omega)^{(k)} \xi_{n1}(\omega)^{(l)}]- E^*[\xi_{n1}(\omega)^{(k)}] E^*[\xi_{n1}(\omega)^{(l)}]\\
&=& n^{-1} \sum_{i=1}^n X_i(\omega)^{(k)}X_i(\omega)^{(l)}-\overline{X}_n(\omega)^{(k)}\overline{X}_n(\omega)^{(l)}
\rightarrow \Gamma_{kl} \quad \text{for } \mathbb{P}-\text{almost all } \omega \in \Omega
\end{eqnarray*}
by the Strong Law of Large Numbers (SLLN). Finally,
\begin{eqnarray} \label{Lindeberg}
0 &\le& L_n(\epsilon,\omega)= \mathbb{E}^*[||X_{n1}(\omega)||^2 1_{\{||X_{n1}(\omega)||> \epsilon \sqrt{n}\}}] \nonumber\\
&=& n^{-1}\sum_{i=1}^n ||X_i(\omega)-\overline{X}_n(\omega)||^2 1_{\{||X_i(\omega)-\overline{X}_n(\omega)||> \epsilon \sqrt{n}\}} \label{LB}\\
&\le& 2 n^{-1} \sum_{i=1}^n ||X_i(\omega)||^2 1_{\{||X_i(\omega)-\overline{X}_n(\omega)||> \epsilon \sqrt{n}\}}+2 ||\overline{X}_n(\omega)||^2, \label{LBbound}
\end{eqnarray}
where the last inequality follows from $||X_i(\omega)-\overline{X}_n(\omega)||^2 \le 2 (||X_i||^2+||\overline{X}_n(\omega)||^2)$. Since
$||\overline{X}_n(\omega)||^2 \rightarrow 0$ for $\mathbb{P}-\text{almost all } \omega \in \Omega$ by the SLLN and continuity of the norm, it remains to consider the first summand in (\ref{LBbound}) without the factor $2$, i.e.
$$
 \overline{L}_n(\epsilon,\omega):=n^{-1} \sum_{i=1}^n ||X_i(\omega)||^2 1_{\{||X_i(\omega)-\overline{X}_n(\omega)||> \epsilon \sqrt{n}\}} \ge 0.
$$
For that purpose notice that for every $m \in \mathbb{N}$ there exists a natural number $n_0=n_0(m)$ such that $\epsilon \sqrt{n} \ge m$ for all
$n \ge n_0$ and therefore
$$
 \overline{L}_n(\epsilon,\omega) \le n^{-1} \sum_{i=1}^n ||X_i(\omega)||^2 1_{\{||X_i(\omega)-\overline{X}_n(\omega)||> m\}} \quad \forall \; n \ge n_0 \; \forall\; m \in \mathbb{N} \; \forall \; \omega \in \Omega.
$$
Thus, if
$$
\Omega_0:= \liminf_{n \rightarrow \infty} \{||\overline{X}_n|| \le 1/2\}= \{||\overline{X}_n|| \le 1/2 \text{ for eventually all } n \in \mathbb{N}\},
$$
then $\mathbb{P}(\Omega_0)=1$ by the SLLN. (Here and in the sequel we omit the argument $\omega \in \Omega$.) Moreover, by $||X_i-\overline{X}_n||\le ||X_i||+ ||\overline{X}_n||$,
\begin{equation} \label{superset}
 \Omega_0 \subseteq \liminf_{n \rightarrow \infty} \{\overline{L}_n(\epsilon) \le n^{-1} \sum_{i=1}^n ||X_i||^2 1_{\{||X_i||> m/2\}}\}.
\end{equation}
Observe that the inclusion in (\ref{superset}) holds for all $m \in \mathbb{N}$, whence
$$
 \Omega_0 \subseteq \bigcap _{m \in \mathbb{N}} \liminf_{n \rightarrow \infty} \{\overline{L}_n(\epsilon) \le n^{-1} \sum_{i=1}^n ||X_i||^2 1_{\{||X_i||> m/2\}}\}.
$$
Consequently,
$$
 \Omega_0 \subseteq \{\limsup_{n \rightarrow \infty}\overline{L}_n(\epsilon) \le \limsup_{n \rightarrow \infty}n^{-1} \sum_{i=1}^n ||X_i||^2 1_{\{||X_i||> m/2\}} \; \forall \; m \in \mathbb{N}\}.
$$
Put
$$
  \Omega_m:=\{n^{-1} \sum_{i=1}^n ||X_i||^2 1_{\{||X_i||> m/2\}} \rightarrow \mathbb{E}[||X_1||^2 1_{\{||X_1||> m/2\}}], n \rightarrow \infty\}.
$$
Another application of the SLLN yields that $\Omega_\infty := \bigcap_{m \in \mathbb{N}} \Omega_m$ has probability $1$.
Let
$$
  E(m):= \mathbb{E}[||X_1||^2 1_{\{||X_1||> m/2\}}].
$$
Then
\begin{equation} \label{superset2}
 \Omega_0 \cap \Omega_\infty \subseteq \{ 0 \le \limsup_{n \rightarrow \infty} \overline{L}_n(\epsilon) \le E(m) \; \forall \; m \in \mathbb{N}\}.
\end{equation}
The Dominated Convergence Theorem ensures that $E(m) \rightarrow 0$ as $m \rightarrow \infty$. Thus taking the limit $m \rightarrow \infty$ in the event on the right side of (\ref{superset2}) shows that
$$
\{ 0 \le \limsup_{n \rightarrow \infty} \overline{L}_n(\epsilon) \le E(m) \; \forall \; m \in \mathbb{N}\} \subseteq \{\overline{L}_n(\epsilon) \rightarrow 0\}.
$$
Deduce from (\ref{superset2}) that $\overline{L}_n(\epsilon) \rightarrow 0 \; \mathbb{P}-$almost surely (a.s.), which in view of (\ref{LBbound}) results in
$L_n(\epsilon) \rightarrow 0$ a.s. for all $0<\epsilon \in \mathbb{Q}.$ Since $\mathbb{Q}$ is countable, we actually have
that $L_n(\epsilon) \rightarrow 0$ for all $0<\epsilon \in \mathbb{Q}$ a.s.
Now, Lemma \ref{convVn} yields the desired result (\ref{bootstrapconv}).
\end{proof}

As a consequence we obtain\\

\begin{corollary} \label{bootworks}(\textbf{Bootstrap works}) Under the assumptions of Theorem \ref{bootstrapfinite} the following uniform convergence holds:
$$
\sup_{x \in \mathbb{R}}|\mathbb{P}^*(V_n^*(\omega)\le x)-\mathbb{P}(V_n \le x)| \rightarrow 0,\; n \rightarrow \infty, \quad \text{for } \mathbb{P}\text{-almost all } \omega \in \Omega.
$$
\end{corollary}

\begin{proof}
The assertion follows from (\ref{uniformVnstar}) and (\ref{uniformVn}) in combination with the triangular inequality.
\end{proof}

Let $F_n$ be the distribution function of $V_n$ and $c_{n,\alpha}=F_n^{-1}(1-\alpha)$ be the exact $(1-\alpha)$-quantile, i.e.
$\mathbb{P}(V_n > c_{n,\alpha}) \le \alpha.$ By Corollary \ref{bootworks} it is reasonable to approximate $c_{n,\alpha}$ by
$c_{n,\alpha}^*(\omega) = (F_n^*(\omega))^{-1}(1-\alpha)$, where $F_n^*(\omega)$ is the distribution function of $V_n^*(\omega)$.
(These quantiles can be approximated with arbitrary accuracy by the Monte-Carlo method.) Thus given a level of significance $\alpha \in (0,1)$, our test rejects $H_0: \mu=0$, if
$V_n > c_{n,\alpha}^*$.\\

\begin{proposition} \label{levelalphatestdfinite} Let $X_i, i \in \mathbb{N},$ be i.i.d. random vectors in $\mathbb{R}^d$ with
$\mathbb{E}[X_1]=0$ and $\mathbb{E}[||X_1||^2] < \infty.$ Then for each $\alpha \in (0,1),$
$$
  \mathbb{P}( H_0 \text{ is rejected at level } \alpha )= \mathbb{P}(V_n > c_{n,\alpha}^*) \rightarrow \alpha.
$$
\end{proposition}

\begin{proof}
Since $F^{-1}$ is continuous, it follows from (\ref{bootstrapconv}) and Proposition 3.1 in Shorack \cite{Shorack} that $c_{n,\alpha}^* \rightarrow c_\alpha$ almost surely. Therefore by (\ref{VCLT}) and Slutsky's theorem,
\begin{equation} \label{Vn-cnstar}
V_n-c_{n,\alpha}^* \stackrel{\mathcal{D}}{\rightarrow} V-c_\alpha.
\end{equation}

Infer with the Portmanteau-theorem that
\begin{eqnarray*}
\alpha &=& \mathbb{P}(V > c_\alpha)\\
&=&\mathbb{P}(V-c_\alpha \in (0,\infty))\\
&\le& \liminf_{n \rightarrow \infty}\mathbb{P}(V_n - c_{n,\alpha}^* \in (0,\infty)) \quad \text{ by } (\ref{Vn-cnstar}), \text{ because } (0,\infty) \text{ is open}\\
&=& \liminf_{n \rightarrow \infty}\mathbb{P}(V_n > c_{n,\alpha}^* )\\
&\le& \limsup_{n \rightarrow \infty}\mathbb{P}(V_n > c_{n,\alpha}^* )\\
&\le& \limsup_{n \rightarrow \infty}\mathbb{P}(V_n \ge c_{n,\alpha}^* )\\
&=& \limsup_{n \rightarrow \infty}\mathbb{P}(V_n - c_{n,\alpha}^* \in [0,\infty))\\
&\le& \mathbb{P}(V - c_\alpha \in [0,\infty)) \quad \quad \text{ by } (\ref{Vn-cnstar}), \text{ because } [0,\infty) \text{ is closed}\\
&=&\mathbb{P}(V \ge c_\alpha) = \mathbb{P}(V > c_\alpha), \quad \text{ because } F \text{ is continuous.}\\
&=&\alpha.
\end{eqnarray*}
\end{proof}

\section{The case $d$ tends to infinity}
So far, the dimension $d$ of our data is a fixed natural number and our test-statistic $V_n$ depends on it: $V_n=V_n(d)$.
Assume that $d=d_n$ tends to infinity: $d_n \rightarrow \infty.$ Then what can be said about the asymptotic behaviour of
$V_n(d_n)$?
Notice that we now observe random vectors $X_{n1},\ldots,X_{nn}$ with values in $\mathbb{R}^{d_n}$. This means that the sample space $\mathbb{R}^{d_n}$ varies with increasing sample sizes $n \in \mathbb{N}$. We overcome this problem by assuming that the $X_{ni}$ arise from a sequence $Z_i, i \in \mathbb{N},$ of i.i.d. random variables with values in the space $l^2$ of all square-summable real sequences as follows:
\begin{equation} \label{arisesequence}
 X_{ni}=(Z_i^{(1)},\ldots,Z_i^{(d_n)}), 1 \le i \le n, \; n \in \mathbb{N}.
\end{equation}
In this section, $||\cdot||_d$ is the euclidian norm on $\mathbb{R}^d$ and contrary to our terminology in the previous section $||\cdot||$
now denotes the $l^2-$norm. With that notation $V_n(d_n)=n ||\overline{X}_n||^2_{d_n}$, where $\overline{X}_n = n^{-1} \sum_{i=1}^n X_{ni}$.
Our next result gives an answer to the question asked above.\\

\begin{proposition} \label{highdimensions} Suppose $Z_i, i \in \mathbb{N},$ are i.i.d. random variables in $l^2$. If $\mathbb{E}[Z_1]=0, \;\mathbb{E}[||Z_1||^2]< \infty$ and $\Gamma := \text{Cov}[Z_1]$, then
\begin{equation} \label{Vnpn}
 \sup_{x \in \mathbb{R}}|\mathbb{P}(V_n(d_n)\le x) - F_\infty(x)| \rightarrow 0, \; n \rightarrow \infty.
\end{equation}
Here, $F_\infty$ is the distribution function of $V_\infty = \sum_{i=1}^\infty \lambda_i N_i^2$, where $\lambda_1,\lambda_2,\ldots$ are the eigenvalues of $\Gamma$ and the $N_i, i \in \mathbb{N},$ are i.i.d. standard normal random variables.
\end{proposition}

\begin{proof} For each natural number $n$ let $c_n:l^2 \rightarrow \mathbb{R}^{d_n}$ be defined by $c_n(x):=(x^{(1)},\ldots,x^{(d_n)})$ for $x=(x^{(i)}:i \in \mathbb{N}) \in l^2.$ Moreover, consider $r_n:\mathbb{R}^{d_n} \rightarrow l^2$ given by $r_n(x^{(1)},\ldots,x^{(d_n)}):=
(x^{(1)},\ldots,x^{(d_n)},0,0,\ldots)$ and the composition $p_n=r_n\circ c_n:l^2 \rightarrow l^2.$ One verifies easily that $c_n, r_n$ and $p_n$ are linear operators and that
\begin{equation} \label{norm}
 ||x||_{d_n}=||r_n(x)|| \quad \forall \; x \in \mathbb{R}^{d_n}.
\end{equation}
Notice that $X_{ni}=c_n(Z_i)$ and
therefore
\begin{equation} \label{embedding}
\overline{X}_n = n^{-1} \sum_{i=1}^n X_{ni} = n^{-1} \sum_{i=1}^n c_n(Z_i) = c_n(n^{-1} \sum_{i=1}^n Z_i)
\end{equation}
by linearity of $c_n$.
So, if $\overline{Z}_n:=n^{-1} \sum_{i=1}^n Z_i$, then using first (\ref{embedding}) and then (\ref{norm}) yields
\begin{equation}
 ||\overline{X}_n||_{d_n}=||c_n(\overline{Z}_n)||_{d_n}=||r_n(c_n(\overline{Z}_n))||=||p_n(\overline{Z}_n)||.
\end{equation}
It follows that
\begin{eqnarray} \label{rep}
V_n(d_n)&=& n||\overline{X}_n||_{d_n}^2 = n ||p_n(\overline{Z}_n)||^2 = ||n^{1/2} p_n(\overline{Z}_n)||^2= ||p_n(n^{1/2} \overline{Z}_n)||^2 \nonumber\\
&=& ||p_n(n^{-1/2} \sum_{i=1}^n Z_i)||^2
\end{eqnarray}
The CLT in Hilbert spaces (confer Proposition 17.29 in Henze \cite{Henze}) says that
\begin{equation} \label{CLTH}
 \frac{1}{\sqrt{n}} \sum_{i=1}^n Z_i \stackrel{\mathcal{D}}{\rightarrow} N(0,\Gamma) \quad \text{in } l^2.
\end{equation}
Now, the functions $p_n$ converge continuously to the identity map $id$. To see this assume that $x_n=(x_n^{(i)}: i \in \mathbb{N})$ converges to some $x=(x^{(i)}: i \in \mathbb{N})$ in $l^2$, i.e. $||x_n-x|| \rightarrow 0$. Observe that
$$
 ||p_n(x_n)-x||^2= \sum_{i=1}^\infty (p_n(x_n)^{(i)}-x^{(i)})^2= \sum_{i=1}^{d_n}(x_n^{(i)}-x^{(i)})^2 + \sum_{i>d_n} (x^{(i)})^2 \le ||x_n-x||^2+ \sum_{i>d_n} (x^{(i)})^2.
$$
Here, the first summand of the upper bound converges to zero by assumption on $(x_n)$. The second summand as the remainder term of
the convergent series $\sum_{i=1}^\infty (x^{(i)})^2$ also converges to zero, because $d_n \rightarrow \infty$.
 Consequently, $p_n(x_n) \rightarrow x$. Thus the CLT (\ref{CLTH}) and the Extended Continuous Mapping Theorem 5.5 of Billingsley \cite{Billingsley} yield that
$$
  p_n(n^{-1/2} \sum_{i=1}^n Z_i) \stackrel{\mathcal{D}}{\rightarrow} N(0,\Gamma) \quad \text{in } l^2.
$$
Infer from the representation (\ref{rep}) and the CMT that
$$
 V_n(d_n) \stackrel{\mathcal{D}}{\rightarrow} ||N(0,\Gamma)||^2.
$$
By Proposition 17.26 in Henze \cite{Henze}, $||N(0,\Gamma)||^2 \stackrel{\mathcal{D}}{=} \sum_{i=1}^\infty \lambda_i N_i^2 = V_\infty$. Since the distribution function $F_\infty$ of $V_\infty$ is continuous, the assertion (\ref{Vnpn}) follows from P\'{o}lya's theorem.
\end{proof}

Just as in the case of fixed dimension $d$ a CLT in $l^2$ for arrays is required. For its formulation and in particularly for its proof we need
the following linear operators: $\gamma_l:l^2 \rightarrow \mathbb{R}^l$ given by $\gamma_l(x)=(x^{(1)},\ldots,x^{(l)}), x=(x^{(i)}: i \in \mathbb{N}), \rho_l:\mathbb{R}^l \rightarrow l^2$ with $\rho(x^{(1)},\ldots,x^{(l)})=(x^{(1)},\ldots,x^{(l)},0,0,\ldots)$ and $\pi_l:= \rho_l\circ \gamma_l: l^2 \rightarrow l^2.$ The latter is the orthoprojector onto the linear hull of the first $l$ unit vectors. For $l=d_n$ we obtain the operators $c_n, r_n$ and $p_n$ from the above proof.\\

\begin{proposition} \label{CLTl2arrays} For every $n \in \mathbb{N}$ let $Z_{n1},\ldots,Z_{nn}$ i.i.d. random variables in $l^2$.
Assume that
\begin{itemize}
\item[(1)] $\mathbb{E}[Z_{n1}^{(i)}]=0$ and $\mathbb{E}[(Z_{n1}^{(i)})^2] < \infty \quad \forall \; i \in \mathbb{N} \; \forall \; n \in \mathbb{N}.$\\

\item[(2)] $\Gamma_n(k,l):= \mathbb{E}[Z_{n1}^{(k)} Z_{n1}^{(l)}] \rightarrow \Gamma(k,l) \in \mathbb{R}, \; n \rightarrow \infty \quad \forall \; k,l \in \mathbb{N}.$\\

\item[(3)] $\limsup_{n \rightarrow \infty} \sum_{k \ge 1} \Gamma_n(k,k) \le \sum_{k \ge 1} \Gamma(k,k) < \infty.$\\

\item[(4)] $\mathbb{E}[||\gamma_l(Z_{n1})||_l^2 1_{\{||\gamma_l(Z_{n1})||_l > \epsilon \sqrt{n}\}}] \rightarrow 0, \; n \rightarrow \infty \quad \text{for all rational } \epsilon>0.$\\
\end{itemize}
Then
\begin{equation} \label{CLTl2}
  \frac{1}{\sqrt{n}} \sum_{i=1}^n Z_{ni} \stackrel{\mathcal{D}}{\rightarrow} N(0,\Gamma) \quad \text{in } l^2,
\end{equation}
where $\Gamma= (\Gamma(k,l))_{k,l \in \mathbb{N}}.$
\end{proposition}

\begin{proof} We use Proposition 17.27 in Henze \cite{Henze}, which gives a sufficient condition for distributional convergence in separable
Hilbert spaces. It is a special case of Proposition 8.5.1 in G\"{a}nssler and Stute \cite{Stute}, which treats more generally separable metric spaces and originally goes back to Proposition 1 of Wichura \cite{Wichura}. Put $$S_n:=n^{-1/2}\sum_{i=1}^n Z_{ni}.$$ Then for all $\delta>0$:
\begin{eqnarray}
\mathbb{P}(||S_n-\pi_l(S_n)|| \ge \delta) &\le& \delta^{-2} \mathbb{E}[||S_n-\pi_l(S_n)||^2]=\delta^{-2} \mathbb{E}[\sum_{k > l} (S_n^{(k)})^2] \nonumber\\
&=& \delta^{-2} \sum_{k > l} \mathbb{E}[(S_n^{(k)})^2] = \delta^{-2} \sum_{k > l} \Gamma_n(k,k). \label{remainderofsum}
\end{eqnarray}
Infer from (3) and (2) that
$$
 \limsup_{n \rightarrow \infty} \sum_{k > l} \Gamma_n(k,k) \le \sum_{k > l} \Gamma(k,k) \quad \text{for all } l \in \mathbb{N}.
$$
Thus (\ref{remainderofsum}) yields
$$
 \limsup_{n \rightarrow \infty}\mathbb{P}(||S_n-\pi_l(S_n)|| \ge \delta) \le \delta^{-2} \sum_{k > l} \Gamma(k,k) \quad \text{for all } l \in \mathbb{N}.
$$
By assumption (3) the remainder of the series converges to zero as $l$ tends to infinity. Therefore we obtain that
\begin{equation} \label{cond1}
 \lim_{l \rightarrow \infty} \limsup_{n \rightarrow \infty}\mathbb{P}(||S_n-\pi_l(S_n)|| =0 \quad \text{for all positive } \delta.
\end{equation}
Furthermore, a routine application of the multivariate Lindeberg-Feller CLT gives
\begin{equation} \label{gammal}
\gamma_l(S_n) \stackrel{\mathcal{D}}{\rightarrow} \gamma_l(N(0,\Gamma)) \quad \text{in } \mathbb{R}^l \quad \text{for all } l \in \mathbb{N}.
\end{equation}
Notice that by Proposition 17.25 in Henze \cite{Henze} finiteness of trace$(\Gamma)$ as required in (3) ensures the existence of $N(0,\Gamma)$ in $l^2.$
Since $||\rho_l(x)- \rho_l(y)||=||\rho_l(x-y)||=||x-y||_l$, it follows that $\rho_l$ is continuous, whence (\ref{gammal}) and the CMT ensure that
\begin{equation} \label{cond2}
\pi_l(S_n) \stackrel{\mathcal{D}}{\rightarrow} \pi_l(N(0,\Gamma)) \quad \text{in } \mathbb{R}^l \quad \text{for all } l \in \mathbb{N}
\end{equation}
upon noticing that $\pi_l = \rho_l \circ \gamma_l$. By (\ref{cond1}) and (\ref{cond2}) we can apply Proposition 17.27 in Henze \cite{Henze}, which gives the desired result.
\end{proof}

Central limit theorems for triangular arrays of Banach space-valued random variables are well established in the literature (see Araujo and Gin\'{e} \cite{Araujo}). However, these results are typically formulated in terms of global norm conditions and abstract operator assumptions. In contrast, Proposition 3 exploits the specific structure of $l^2$
and replaces global conditions by componentwise second-moment assumptions, convergence of finite-dimensional covariances, and a trace-type condition.
We would like to emphasize that our CLT in $l^2$ represents a new result in that it cannot be readily derived from the general theory of Araujo and Gin\'{e} \cite{Araujo}.\\


\begin{corollary} \label{normZnquer} Let $\overline{Z}_n := n^{-1} \sum_{i=1}^n Z_{ni}$ and $W_n:=n||\overline{Z}_n||^2.$ Then under the assumption of the above Proposition \ref{CLTl2arrays},
\begin{equation} \label{Vnpn}
 \sup_{x \in \mathbb{R}}|\mathbb{P}(W_n \le x - F_\infty(x)| \rightarrow 0, \; n \rightarrow \infty.
\end{equation}
Here, $F_\infty$ is the distribution function of $V_\infty = \sum_{i=1}^\infty \lambda_i N_i^2$, where $\lambda_1,\lambda_2,\ldots$ are the eigenvalues of $\Gamma$ and the $N_i, i \in \mathbb{N},$ are i.i.d. standard normal.
\end{corollary}

\begin{proof} $W_n=||S_n||^2 \stackrel{\mathcal{D}}{\rightarrow} ||N(0,\Gamma)||^2$ by (\ref{CLTl2}) and the CMT. Since $||N(0,\Gamma)||^2 \stackrel{\mathcal{D}}{=} V_\infty$ and $F_\infty$ is continuous, the assertion follows from P\'{o}lya's theorem.
\end{proof}

\section{The bootstrap when d tends to infinity}
Recall that we observe $X_{n1},\ldots,X_{nn}$ in $\mathbb{R}^{d_n}$ with
$$
 X_{ni}= (X_{ni}^{(1)},\ldots,X_{nn}^{(d_n)}) = (Z_i^{(1)},\ldots,Z_i^{(d_n)}), \; 1 \le i \le n.
$$

With the random indices $i_{n1},\ldots,i_{nn}$ we introduce as in section 1 $\xi_{nk}(\omega):\Omega^* \rightarrow \mathbb{R}^{d_n}$ by
$$
 \xi_{nk}(\omega)(\omega^*):= X_{n i_{nk}(\omega^*)}(\omega)
$$
and with $\overline{X}_n(\omega)=n^{-1} \sum_{i=1}^n X_{ni}(\omega)$:
$$
 \overline{\xi}_n(\omega):= n^{-1} \sum_{k=1}^n (\xi_{nk}(\omega)-\overline{X}_n(\omega)).
$$
Then the bootstrap version of $V_n(d_n)$ is $$W_n^*(\omega):=n||\overline{\xi}_n(\omega)||_{d_n}^2.$$

\begin{theorem} \label{bootstraphigh} Suppose $Z_i, i \in \mathbb{N},$ are i.i.d. random variables with values in $l^2$ such that
$\mathbb{E}[Z_1]=0$ and  $\mathbb{E}[||Z_1||^2] < \infty$. If $\Gamma = \text{Cov}[Z_1]$, then
\begin{equation} \label{uniformWnstar}
  \sup_{x \in \mathbb{R}}|\mathbb{P}^*(W_n^*(\omega) \le x)-F_\infty(x)| \rightarrow 0, \;  n \rightarrow \infty,\quad \text{for } \mathbb{P}\text{-almost all } \omega \in \Omega.
\end{equation}
Here, the $\lambda_i, i \in \mathbb{N}$ are the eigenvalues of $\Gamma$ and $N_i, i \in \mathbb{N}$ are i.i.d. standard normal. Moreover,  $F_\infty$ is the distribution function of $V_\infty=\sum_{i=1}^\infty \lambda_i N_i^2.$
\end{theorem}

\begin{proof} For each $n \in \mathbb{N}$ let $$\zeta_{nk}(\omega)(\omega^*):=Z_{i_{nk}(\omega^*)}(\omega), \; 1 \le k \le n$$ and
$$
\overline{\zeta}_n(\omega):= \frac{1}{n} \sum_{k=1}^n (\zeta_{nk}(\omega)- \overline{Z}_n(\omega))
$$
with $\overline{Z}_n(\omega):=n^{-1} \sum_{i=1}^n Z_i(\omega).$ Then by linearity and (\ref{arisesequence}), $\overline{\xi}_n(\omega) = c_n(\overline{\zeta}_n(\omega))$ and
therefore $||\overline{\xi}_n(\omega)||_{d_n}=||c_n(\overline{\zeta}_n(\omega))||_{d_n}$, which by (\ref{norm}) results in
$$
 W_n^*(\omega)=n||p_n(\overline{\zeta}_n(\omega))||^2.
$$
Since $p_n(\zeta_n(\omega))=n^{-1} \sum_{i=1}^n p_n(\zeta_{ni}(\omega)-\overline{Z}_n(\omega))$, we want to apply Corollary \ref{normZnquer} to
the array $Z_{ni}(\omega):= p_n(\zeta_{ni}(\omega)-\overline{Z}_n(\omega)), 1 \le i \le n, n \in \mathbb{N}.$ For that purpose notice that
$\zeta_{n1}(\omega),\ldots,\zeta_{nn}(\omega)$ are i.i.d. under $\mathbb{P}^*$ with common distribution
$$
 Q_n(\omega)=\frac{1}{n} \sum_{i=1}^n \delta_{Z_i(\omega)},
$$
where $\delta_x$ is the Dirac-measure at point $x \in l^2.$ Thus, $Z_{n1}(\omega),\ldots,Z_{nn}(\omega)$ are i.i.d. under $\mathbb{P}^*$ for every $n \in \mathbb{N}.$ Furthermore,
$$
 \mathbb{E}^*[Z_{n1}(\omega)]=\mathbb{E}^*[p_n(\zeta_{n1}(\omega)-\overline{Z}_n(\omega))]
 =p_n(\mathbb{E}^*[\zeta_{n1}(\omega)-\overline{Z}_n(\omega)])=p_n(0)=0,
$$
because by there definitions the operators $\mathbb{E}^*$ and $p_n$ can be commuted and $\mathbb{E}^*[\zeta_{n1}(\omega)]= \overline{Z}_n(\omega).$
In particular
\begin{equation} \label{Zn1arecentered}
\mathbb{E}^*[Z_{n1}(\omega)^{(i)}] = 0 \quad \forall \; i,n \in \mathbb{N} \; \forall \; \omega \in \Omega.
\end{equation}
In the sequel notice that by definition
\begin{equation} \label{zerocomponents}
   Z_{n1}(\omega)^{(i)} = \zeta_{n1}(\omega)^{(i)}-\overline{Z}_n(\omega)^{(i)} \quad \forall \; i \le d_n \quad \text{and} \quad                           Z_{n1}(\omega)^{(i)} =0 \quad \forall \; i >d_n.
\end{equation}
So, if $i \le d_n$, then
$$
 \mathbb{E}^*[(Z_{n1}(\omega)^{(i)})^2]=\int_{l^2} (x^{(i)}-\overline{Z}_n(\omega)^{(i)})^2 Q_n(\omega)(dx)= n^{-1} \sum_{j=1 }^n (Z_j(\omega)^{(i)}-\overline{Z}_n(\omega)^{(i)})^2,
$$
i.e. the second moment is equal to the sample variance. With a view to (\ref{zerocomponents}) we obtain that
\begin{equation} \label{Zn1secondmoments}
\mathbb{E}^*[(Z_{n1}(\omega)^{(i)})^2] < \infty \quad \forall \; i,n \in \mathbb{N} \; \forall \; \omega \in \Omega.
\end{equation}
Similarly, for every fixed natural numbers $k,l \le d_n$ it follows that
\begin{eqnarray*}
 \Gamma_n(\omega)(k,l)&:=& \text{Cov}^*[Z_{n1}(\omega)^{(k)}, Z_{n1}(\omega)^{(l)}]=\text{Cov}^*[\zeta_{n1}(\omega)^{(k)}-\overline{Z}_n(\omega)^{(k)},\zeta_{n1}(\omega)^{(l)}-\overline{Z}_n(\omega)^{(l)}]\nonumber\\
 &=&\text{Cov}^*[\zeta_{n1}(\omega)^{(k)},\zeta_{n1}(\omega)^{(l)}]\nonumber\\
 &=& n^{-1} \sum_{i=1}^n Z_i(\omega)^{(k)} Z_i(\omega)^{(l)}- \overline{Z}_n(\omega)^{(k)} \overline{Z}_n(\omega)^{(l)}.
\end{eqnarray*}
Since $k,l \le d_n$ for eventually all $n \in \mathbb{N}$ as $d_n \rightarrow \infty$, the SLLN yields that
\begin{equation} \label{convergenceofthecovariances}
 \Gamma_n(\omega)(k,l) \rightarrow (\text{Cov}[Z_1])_{k,l} \quad \forall \; k,l \in \mathbb{N} \quad \text{for }\mathbb{P}-\text{almost all } \omega \in \Omega.
\end{equation}

Observe that by the Dominated Convergence Theorem,
\begin{eqnarray}
 \sum_{k \ge 1}\Gamma_n(\omega)(k,k)&=& \sum_{k \ge 1}\mathbb{E}^*[Z_{n1}(\omega)^2]= \mathbb{E}^*[||Z_{n1}(\omega)||^2]=\mathbb{E}^*[||p_n(\zeta_{n1}(\omega))-p_n(\overline{Z}_n(\omega)||^2] \nonumber\\
 &=& n^{-1}\sum_{i=1}^n ||p_n(Z_i(\omega))-p_n(\overline{Z}_n(\omega)||^2. \label{sumnormsquared}
\end{eqnarray}
Let $<\cdot,\cdot>$ denote the scalar product in $l^2$. Then
$$
||p_n(Z_i(\omega))-p_n(\overline{Z}_n(\omega)||^2 = ||p_n(Z_i(\omega))||^2 - 2<p_n(Z_i(\omega)),p_n(\overline{Z}_n(\omega)>+||p_n(\overline{Z}_n(\omega)||^2
$$
Thus (\ref{sumnormsquared}) and linearity of the scalar product yield
\begin{equation} \label{upperboundsumGamman}
 \sum_{k \ge 1}\Gamma_n(\omega)(k,k)= n^{-1}\sum_{i=1}^n ||p_n(Z_i(\omega))||^2-||p_n(\overline{Z}_n(\omega)||^2 \le  n^{-1}\sum_{i=1}^n ||Z_i(\omega))||^2,
\end{equation}
because $||p_n(x)|| \le ||x||$ for all $x \in l^2.$ By the SLLN the upper bound in (\ref{upperboundsumGamman}) converges to $\mathbb{E}[||Z_1||^2]$ for $\mathbb{P}-$almost every $\omega \in \Omega$. Since $\mathbb{E}[||Z_1||^2]=\sum_{k \ge 1}\Gamma(k,k)$ by the Monotone Convergence Theorem, we therefore obtain
\begin{equation} \label{tracecondition}
 \limsup_{n \rightarrow \infty}\sum_{k \ge 1}\Gamma_n(\omega)(k,k) \le \sum_{k \ge 1}\Gamma(k,k)<\infty \quad \text{for }\mathbb{P-}\text{almost all } \omega \in \Omega,
\end{equation}
where finiteness of the series follows from $\mathbb{E}[||Z_1||^2]< \infty$ by assumption. Finally, notice that
$\gamma_l(Z_{n1}(\omega))= \gamma_l(\zeta_{n1}(\omega))-\gamma_l(\overline{Z}_n(\omega)) \in \mathbb{R}^l$ for eventually all $n \in \mathbb{N}$.
For all these $n$ the Lindeberg-term is equal to
\begin{eqnarray*}
 L_n(\epsilon,\omega)&=&\mathbb{E}^*[||\gamma_l(Z_{n1}(\omega))||^2 1_{\{||\gamma_l(Z_{n1}(\omega))||> \sqrt{n} \epsilon\}}\\
  &=& n^{-1}\sum_{i=1}^n ||\gamma_l(Z_i(\omega))-\gamma_l(\overline{Z}_n(\omega))||^2 1_{||\gamma_l(Z_i(\omega))-\gamma_l(\overline{Z}_n(\omega))||>\sqrt{n} \epsilon\}}.
\end{eqnarray*}
This is equation (\ref{LB}) with $X_i$ there corresponds to $\gamma_l(Z_i)$ here. From this point (\ref{LB}) on, we can use the same line of reasoning to obtain
\begin{equation} \label{LBcondition}
 L_n(\epsilon,\omega) \rightarrow 0 \quad \text{for all rational } \epsilon >0 \quad \text{for }\mathbb{P}\text{-almost all } \omega \in \Omega.
\end{equation}

With (\ref{Zn1arecentered}), (\ref{Zn1secondmoments}), (\ref{convergenceofthecovariances}), (\ref{tracecondition}) and (\ref{LBcondition})
all requirements of Corollary \ref{normZnquer} are fulfilled for $\mathbb{P}$-almost all $\omega \in \Omega$, which yields the desired result.
\end{proof}

Proposition \ref{highdimensions} and Theorem \ref{bootstraphigh} with the triangular inequality yield:\\

\begin{corollary} \label{bootworkshigh}(\textbf{Bootstrap works in high dimensions}) Under the assumptions of Theorem \ref{bootstraphigh} the following uniform convergence holds:
$$
\sup_{x \in \mathbb{R}}|\mathbb{P}^*(W_n^*(\omega)\le x)-\mathbb{P}(V_n(d_n) \le x)| \rightarrow 0,\; n \rightarrow \infty, \quad \text{for } \mathbb{P}\text{-almost all } \omega \in \Omega.
$$
\end{corollary}

Let $G_n$ be the distribution function of $V_n(d_n)$ and $d_{n,\alpha}=G_n^{-1}(1-\alpha)$ be the exact $(1-\alpha)$-quantile, i.e.
$\mathbb{P}(V_n(d_n) > d_{n,\alpha}) \le \alpha.$ By Corollary \ref{bootworkshigh} it is reasonable to approximate $d_{n,\alpha}$ by
$d_{n,\alpha}^*(\omega) = (G_n^*(\omega))^{-1}(1-\alpha)$, where $G_n^*(\omega)$ is the distribution function of $W_n^*(\omega)$.
These quantiles can be approximated with arbitrary accuracy by the Monte-Carlo method. Thus given a level of significance $\alpha \in (0,1)$, our test rejects $H_0: \mu=0$, if $V_n(d_n)$ exceeds the critical value $d_{n,\alpha}^*$.\\

\begin{proposition} \label{levelalphatestdinfinite} Let $Z_i, i \in \mathbb{N},$ be i.i.d. random variables in $l^2$ with
$\mathbb{E}[X_1]=0$ and $\mathbb{E}[||X_1||^2] < \infty.$ Then for each $\alpha \in (0,1),$
$$
  \mathbb{P}( H_0 \text{ is rejected at level } \alpha )= \mathbb{P}(V_n(d_n) > d_{n,\alpha}^*) \rightarrow \alpha.
$$
\end{proposition}

\begin{proof} Since $F_\infty^{-1}$ is continuous, Theorem \ref{bootstraphigh} and Proposition 3.1 in Shorack guarantee that $d_{n,\alpha}^* \rightarrow d_\alpha := F_\infty^{-1}(1-\alpha)$ almost surely. Conclude with Proposition \ref{highdimensions} and Slutsky's theorem that
$V_n(d_n)- d_{n,\alpha}^*$ converge in distribution to $V_\infty-d_\alpha.$ Now, the assertion follows in the same manner as in the proof of
Proposition \ref{levelalphatestdfinite}.
\end{proof}





\end{document}